\documentclass[12pt,a4paper]{amsart}
\usepackage{amsfonts}
\usepackage{amsthm}
\usepackage{amsmath}
\usepackage{amscd}
\usepackage[latin2]{inputenc}
\usepackage{t1enc}
\usepackage[mathscr]{eucal}
\usepackage{indentfirst}
\usepackage{graphicx}
\usepackage{graphics}
\usepackage{pict2e}
\usepackage{epic}
\numberwithin{equation}{section}
\usepackage[margin=2.9cm]{geometry}
\usepackage{epstopdf} 
\usepackage{hyperref}
\hypersetup{
    colorlinks = true,
    urlcolor   = blue,
    citecolor  = blue,
}

\theoremstyle{plain}
\newtheorem{Th}{Theorem}[section]
\newtheorem{Lem}[Th]{Lemma}

\newtheorem{Prop}[Th]{Proposition}

\newtheoremstyle{TheoremNum}
        {\topsep}{\topsep}         
        {\itshape}                      
        {}                                  
        {\bfseries}                     
        {.}                                 
        { }                                 
        {\thmname{#1}\thmnote{ \bfseries #3}}
\theoremstyle{TheoremNum}
\newtheorem{Thmn}{Theorem}

\theoremstyle{definition}
\newtheorem{Def}[Th]{Definition}
\newtheorem{Conj}[Th]{Conjecture}
\newtheorem{Rem}[Th]{Remark}
\newtheorem{?}[Th]{Problem}

\newcommand{\ns}{\mathcal{N S}(S_g)}
\newcommand{\dns}{d_{\text NS}}

\begin{document}

\title{Origami edge-paths in the curve graph}

\author{Hong Chang, Xifeng Jin, William W. Menasco}

\address{
\noindent
Hong Chang, hchang24@buffalo.edu, Department of Mathematics, University at Buffalo--SUNY
\newline
\indent
Xifeng Jin, xifengji@buffalo.edu, Department of Mathematics, University at Buffalo--SUNY
\newline
\indent
William W. Menasco, menasco@buffalo.edu, Department of Mathematics, University at Buffalo--SUNY} 



\begin{abstract}
An {\em origami} (or flat structure) on a closed oriented surface, $S_g$, of genus $g \geq 2$ is obtained from a finite collection of unit Euclidean squares by gluing each right edge to a left one and each top edge to a bottom one.
The main objects of study in this note are \emph{origami pairs of curves}---filling pairs of simple closed curves, $ (\alpha,\beta)$, in $S_g$ such that their minimal intersection is equal to their algebraic intersection---they are \emph{coherent}.  
An origami pair of curves is naturally associated with an origami on $S_g$.
Our main result establishes that for any origami pair of curves there exists an \emph{origami edge-path}, a sequence of curves, $\alpha =\alpha_0, \alpha_1, \alpha_2, \cdots, \alpha_n = \beta $, such that: $\alpha_i$ intersects $\alpha_{i+1}$ exactly once; any pair $(\alpha_i, \alpha_j)$ is coherent; and thus, any filling pair, $(\alpha_i, \alpha_j)$, is also an origami.
With their existence established, we offer shortest origami edge-paths as an area of investigation.
\end{abstract}
\maketitle


\section{Introduction.}


\subsection{Main results.}
\label{main results}
For this note we will consider a variant of Harvey's \cite{Harvey} curve graph for $S_g$, a closed oriented surface of genus $g \geq 2$.
To start, we will focus on \emph{non-separating} simple closed curves in $S_g$---homotopically non-trivial embedded curves whose complement in $S_g$ is connected.
Let the \emph{non-separating curve graph} $\ns$ be the graph whose vertices are isotopy classes of non-separating simple closed curves; and,
two vertices of $\ns$ will share an edge if they have curve representatives that are disjoint.  
If each edge is endowed with length one, $\ns$ is a metric space---the length between two vertices corresponds to the number of edges in the shortest edge-path
between them.  We will use $\dns (u,v)$ as the notation for the metric function giving us the distance between $u$ and $v$. The non-separating curve graph $\ns$ has been studied in \cite{FM, Ha, Ra-1, Sc}.

Let $\alpha ,\beta \subset S_g$, be two non-separating curves representing vertices $u,v \in \ns$, respectively.  We assume $|\alpha \cap \beta|$ to be minimal over all isotopic
representatives of, say $u$.
We will abuse notation by having $|u \cap v| := |\alpha \cap \beta|$.
The pair $(\alpha , \beta)$ is a \emph{filling pair} when all components of $S_g \setminus (\alpha \cup \beta)$ are topological open discs.
For convenience, we will also say $(u,v)$ is a \emph{filling pair} when $(\alpha, \beta)$ is a filling pair.
An assignment of orientation to both curves allows one to calculate the algebraic intersection number of the curves.  If the absolute value of their
algebraic intersection number is equal to $|\alpha \cap \beta|$ then $(\alpha, \beta)$ is a \emph{coherent pair}.
Again for convenience, we will also say $(u,v)$ is a \emph{coherent pair} if $(\alpha, \beta)$ is a coherent pair.

An \emph{origami} for a closed surface $S_{g \geq 2}$ is obtained from a finite number of Euclidean squares by gluing each right edge to a left one and each top edge to a bottom one.  Our
first result, which has also been independently shown by Jeffreys \cite{Jef}, establishes a correspondence between a coherent filling pair and an origami on $S_g$.

\begin{Th}
\label{ori}
A coherent filling pair of curves naturally corresponds to an origami on $S_g$. 
\end{Th}

Accordingly, we will call a coherent filling pair an \emph{origami pair of curves}.  Similarly, the two associated vertices in $\ns$
correspond to an \emph{origami pair of vertices}.

\begin{Def}
A sequence of vertices, $ \mathcal{E} = \{v_0 , \cdots , v_n \}$, in $\ns$ is an \emph{origami edge-path of length $n$} if:
\begin{itemize}
\item[1.] $ |v_i \cap v_{i+1}| = 1$ for all $0 \leq i \leq n-1$;
\item[2.] any pair $(v_i , v_j)$, $|i-j|>0$, is a coherent pair;
\item[3.] and thus, any filling pair $(v_i , v_j)$ is an origami pair for $S_g$.
\end{itemize}
\end{Def}

At times we will discuss origami edge-paths without reference to their lengths.  When we discuss the length of an origami edge-path, $\mathcal{E}$, of unspecified length we will use, $\|\mathcal{E}\|$, to denote its length.

We have the following main result.

\begin{Th}
{\label{origami paths exist}}
For any origami pair $(u,v)$, there exists an origami edge-path $ \mathcal{E} = \{u = v_0 , \cdots , v_n = v \}$.
\end{Th}

Our proof of Theorem \ref{origami paths exist} uses a ``bicorn'' construction that allows us to say a neighborhood of any geodesic of width $14$ will
contain an origami edge-path.  (See Proposition \ref{14-bounded}.)  Additionally, such origami edge-paths can be arbitrarily long.

\begin{Th}
{\label{infinite diameter}}
Let $N >0$ be any positive integer.  Then there exists an origami pair, $(u,v)$, for $S_g$ such that $ \dns (u,v) > N$.
In particular, for a fixed $u \in \ns$, the set of $ v \in \ns$ such that $(u,v)$ is an origami pair has infinite diameter in $\ns$.
\end{Th}


\subsection{Interesting examples.}
In the literature, there are ready examples of origami pairs of curves for all genera.  For Euler characteristic reasons, the theoretical minimal number of intersections
for a filling pair of $S_g$ is $2g-1$.  For $g=2$ this minima is not geometrically realizable and $4$ is the minimal number of intersections needed for filling.
However, Aougab and Huang showed in \cite{AH} that for all $g >2$ there exists filling pairs of curves whose intersection achieves the $2g-1$ minima. Moreover, in \cite{AMN, CM, Jef} constructions are given for origami pairs of curves that achieve the $2g-1$ minima, for all $g>2$. 
In \S \ref{examples} we briefly describe such examples and their associated origami edge-paths---all are of length 2.

For the $g=2$ case we consider an example of Hempel \cite{Sc}.  The origami filling pair in $S_2$ is distance $4$ in $\ns$ and the length of the origami
edge-path coming from the construction in the proof of Theorem \ref{origami paths exist} is also of length $4$.


\subsection{Background on the curve graph and origamis.}
The \emph{curve complex} of a closed oriented surface $S_g$, $\mathcal{C} (S_g)$, is a simplicial flag complex such that each vertex is represented by the isotopy class of an essential simple closed curve and two vertices are connected by an edge if they have disjoint representatives up to isotopy. It was introduced by Harvey \cite{Harvey} to study the mapping class group of the surface. The 1-skeleton of the curve complex, $\mathcal{C}^1(S_g) \subset \mathcal{C} (S_g)$, is called the \emph{curve graph}.  If each edge is endowed with length one, the curve graph is a metric space. A metric space is called \emph{$\delta$-hyperbolic} if for each geodesic triangle, each side lies in a $\delta$-neighborhood of the other two sides for some $\delta \geq 0$. Masur and Minsky's seminal result in \cite{MM1}
showed that the curve graph is $\delta$-hyperbolic.  The non-separating curve graph $\ns \subset \mathcal{C}^1(S_g)$ as a subgraph, is also $\delta$-hyperbolic \cite{Ha, MS}.  The reason is that the natural inclusion is a quasi-isometry. Similar to 
$\mathcal{C}^1(S_g)$, for $\ns$ the constant $\delta$ can be chosen independent of the genus of the surface---that is we have \emph{uniform hyperbolicity}.

\begin{Th}\emph{(Rasmussen \cite{Ra-1} Theorem 1.1)}\label{Ra}
The non-separating curve graph $\ns$  is connected and uniformly $\delta$-hyperbolic with infinite diameter. 
\end{Th}

The definition of the non-separating curve graph in the theorem differs from the one defined previously.  Rasmussen's add-in is the edges between pairs of curves that intersect at most twice.
As such, a case may be made that such an enhanced $\ns$ is the more natural setting to consider origami edge-paths, since a shortest origami edge-path might naturally
be a geodesic in the Rasmussen-type non-separating graph of curves.

Bicorn technology is used in the proof of Theorem \ref{Ra}.  The reader may also see \cite{Ra-2} for further development of the bicorn technology.

Origamis are of significant interest in Teichm\"uller theory.  (See \cite{DGZZ, Jef, Herr, CMc, Sn}.)  Specifically, an origami determines a Riemann surface along with a translation structure.
The translation structure can be varied in a natural way resulting in a complex one-parameter family of Riemann surfaces.
One can identify this parameter space with the hyperbolic upper half plane, $\mathbb{H}^2$, in such a manner that the Riemann surfaces inherits a natural marking.
Thus, an origami corresponds to a map of $\mathbb{H}^2$ into a Teichm\"uller space.
Moreover, the map is an isometric and holomorphic embedding of $\mathbb{H}^2$---commonly
known as a \emph{Teichm\"uller disc}.

The combinatorial structure of origamis specific to a coherent filling pair of curves---a $[1,1]$-origami---is a current area of investigation and
we recommend the excellent discussion in the introduction of \cite{Jef}.

\subsection{Conjecture for distance calculations.}
\label{conjectures}
Since the seminal work of Masur and Minsky \cite{MM1, MM2}, understanding filling pairs and their behavior with respect to intersection number and distance in the curve graph 
has been the focus of considerable research interest.  Utilizing the Masur-Minsky tight-geodesic technology, the work of Bell and Webb \cite{BW, Webb} describes a polynomial-time algorithm for computing distance in the curve complex.  In \cite{BMW}, efficient geodesics are introduced and an implementation of a distance algorithm for short distances based upon efficient geodesics is
given in \cite{GMMM, soft}.  Additionally, in \cite{JM} the initial intersection bounds with reference arcs that are given for any curve representing a vertex of a geodesic are substantially
improved so that these bounds are independent of distance---\emph{super efficiency}---and only dependent on genus.  Still such distance calculations are
still arduous and one might hope that there would be a subclass of filling pairs that such calculations are readily made.
With the added feature of having a filling pair being coherent, we propose a series of conjectures for calculating ``distance'' for origami pairs of curves.

More specially, it is natural to consider the minimal length of origami edge-paths associated with an origami pair of curves.
That is, let $${\bf E}(u,v) := min \{ \  \|\mathcal{E}\| \  : \ \mathcal{E} {\rm \ is \ an \ origami \ edge\mbox{-}path \ for\ origami \ pair} (u,v)\}$$
be the \emph{origami length of} $(u,v)$.

\begin{Rem}\label{non-transitivity}
The origami length defined above is not a real distance function, as being coherent is not a transitive property between curves.
\end{Rem}

Our bicorn-construction of an origami edge-path gives us that they are within a $14$-neighborhood of a geodesic.
Our first conjecture proposes that origami edge-paths of length ${\bf E}(u,v)$ are well behaved with respect to distance.

\begin{Conj}
\label{Conj1}
Shortest origami edge-paths---that is, having length ${\bf E}(u,v)$---are quasi-geodesics.
\end{Conj}

Moreover, we propose a test for when any origami edge-path is a quasi-geodesic.  
\begin{Conj}
\label{Conj2}
For an origami edge-path, $\{ \alpha = \alpha_0 , \cdots , \alpha_{n>2} = \beta \}$, if the following relationship for intersection quotients holds:
$$\frac{| \beta \cap \alpha_1 |}{| \alpha \cap \alpha_1 |} > \frac{| \beta \cap \alpha_2 |}{| \alpha \cap \alpha_2 |} > \cdots > \frac{| \beta \cap \alpha_n |}{| \alpha \cap \alpha_n |},$$
then the edge-path is a quasi-geodesic.
\end{Conj}
The reader should observe that when $n = 2$, by the definition of an origami edge-path we have $\frac{| \beta \cap \alpha_1 |}{| \alpha \cap \alpha_1 |} = 1$.

To give some motivation for decreasing intersection quotients, one can think of $\frac{| \beta \cap \alpha_i |}{| \alpha \cap \alpha_i |}$ as the slope of the $\alpha_i$
in the $(\alpha, \beta)$ origami (flat) structure.  (Alternatively, it is the slope coming from the $\alpha$-horizontal and $\beta$-vertical measured foliations.)
Intuitively the sequence is strictly decreasing since we want our $\alpha_i$ curves to ``swing'' from being close to ``parallel'' with $\alpha$ to being close to ``parallel'' with $\beta$.
One would expect that if such a swing from $\alpha$ to $\beta$ corresponds to a quasi-geodesic then the edge-path would not
illustrate any ``zig-zagging'' in a bounded neighborhood of a geodesic.  Having the slope of vertices in the edge-path being strictly decreasing would
eliminate such zig-zagging.
 
This intuitive thinking also points to our last conjecture.
\begin{Conj}
\label{Conj3}
For shortest origami edge-paths, $| \beta \cap \alpha_{i + 1} |$ is minimal over all possible curves $\alpha_{i+1}$ intersecting $\alpha_i$ once.
\end{Conj}
Thus, we are conjecturing the existence of a \emph{greedy algorithm} for finding a shortest origami edge-path.


\subsection{Outline of paper.}

For completeness, in \S \ref{origamis} we include a proof of Theorem \ref{ori}.  
In \S \ref{bicorns} we prove a proposition which implies Theorem \ref{origami paths exist}.  As previously mentioned, the proof is constructive and
the containment of origami edge-paths within a $14$ width neighborhood of any geodesic is also argued in Proposition \ref{14-bounded}.
In \S \ref{examples} we discuss examples.  In particular, there exist origami pairs of curves for all genera. 
In \S \ref{arbitrary distance} we use the existence of origamis pairs of curves for all genera coupled with a pseudo-Anosov map construction of Thurston's to prove Theorem \ref{infinite diameter}.


\section{$[1,1]$-origamis and proof of Theorem \ref{ori}.}
\label{origamis}
There are a number of equivalent definitions of an origami of genus $g$, $\mathcal{O}_g$, in the literature, and the most prevalent one defines $\mathcal{O}_g$
as a ramified cover of $S_g$ over the standard torus with all ramification points located over a single point of the torus
However, for ease of argument in our proof of Theorem \ref{ori} we will adopt the following definition \cite{Loch, Sn}.

\begin{Def}
An \emph{origami} consists of a finite set of copies of the unit Euclidean square that are glued together observing the following rules:
\begin{itemize}
\item[$\bullet$] Each left edge of a square is identified by a translation with a right edge.
\item[$\bullet$] Each top edge is identified by a translation with a bottom one.
\item[$\bullet$] The closed (topological) surface that one obtains is connected.
\end{itemize}
\end{Def}

By restricting the gluing rules to only the left-edge-to-right-edge (respectively, upper-edge-to-lower-edge) identifications we obtain a collection of embedded \emph{horizontal annuli}
(respectively, \emph{vertical annuli}).
A \emph{$[1,1]$-origami} will be an origami that has exactly one horizontal annulus and one vertical annulus. In particular, we have the following well known result associating a $[1,1]$-origami
to an origami pair of curves.

\begin{Thmn}[\ref{ori}]
A coherent filling pair of curves(origami pair of curves) naturally corresponds to an origami on $S_g$.
\end{Thmn}

\begin{figure}[ht]
\scalebox{.60}{\includegraphics[origin=c, width=1.50 \textwidth,  height = 1.0 \textwidth]{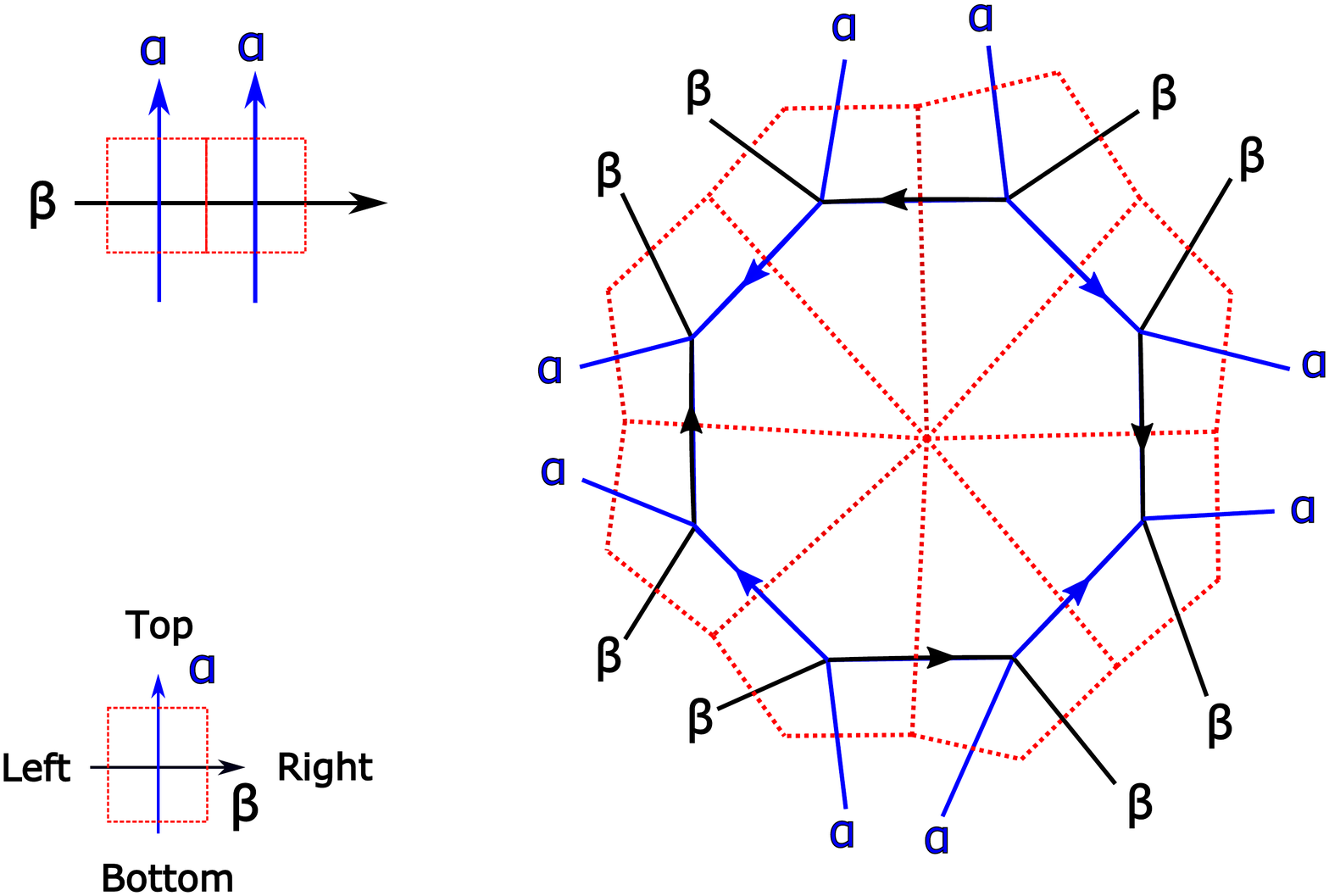}}
\vspace{0.5cm}
\caption{On the left, two unit squares cover two adjacent intersection points along $\beta$, with the center of each square as the intersection point; On the right, 8 unit squares cover the octagon obtained by alternating $\alpha$ and $\beta$-subarcs. The center of each square is the intersection point of $\alpha$ and $\beta$, and one corner point of the 8 squares are  identified as the center of the octagon.}
\label{dual}
\end{figure}

\begin{proof}[Proof of Theorem \ref{ori}]
Let $\alpha$, $\beta$ be an origami pair of curves, then $\alpha\cup\beta$ induces a cellular decomposition on $S_g$.  More precisely, $0$-cells are the intersection points $\alpha\cap\beta$, $1$-cells are the arcs $\alpha\cup\beta-\alpha\cap\beta$, and $2$-cells are $S_g-\alpha\cup\beta$. Note that the $2$-cells are $4n$-gons. This cellular decomposition admits a dual cellular decomposition as follows. The centers of $4n$-gons are the co-vertices, and each arc of $\alpha\cup\beta-\alpha\cap\beta$ is crossed by one co-edge. Each intersection point of $\alpha\cap\beta$ is the center of a co-cell, which is a square as $\alpha$ and $\beta$ pass through it exactly once. 

As a result of the dual cellular decomposition, the surface is covered by squares. The number of squares is equal to the intersection number of $\alpha\cap\beta$. Think of each square as a unit square, see Fig. \ref{dual} for an illustration for the unit squares covering two consecutive intersection points on $\beta$ and an octagon. On the right, the squares covering the octagon are distorted, but they still perfectly cover it.  The center of the octagon is the identification of corners from 8 unit squares. A unit square is bounded by four red dashed segments, each of which is a co-edge of the dual cellular decomposition. 

Since each intersection point of $\alpha$ and $\beta$ is the center of a square, each unit square can be induced with an orientation from that on the intersection point. For example, we can regard the side on the left of $\alpha$ as left, the side on the right of $\alpha$ as right, the side on the left of $\beta$ as top, and the side on the right of $\beta$ as bottom. Since $\alpha$, $\beta$ is a coherent filling pair, traversing along $\beta$, we can observe a left side is glued to a right one. While traversing along $\alpha$, we can see a bottom side is glued to a top one. The intersection number of $\alpha$ and $\beta$
is the same as the number of squares, so traveling along $\alpha$ or $\beta$ exhausts all squares. This induces an $[1,1]$-origami on the surface. In this case, $\alpha$ and $\beta$ are the cores of vertical and horizontal annuli, respectively. 
\end{proof}


\section{Bicorn curves and proof of Theorem \ref{origami paths exist}.}
\label{bicorns}

In this section we prove the statement of Theorem \ref{origami paths exist}, that for any origami pair of curves there is a sequence of curves that are mutually coherent with each other, i.e. the existence an origami edge-path between an origami pair of curves. The proof relies on the notion of bicorn curves introduced by Przytycki and Sisto \cite{PS}. More specifically, the existence of an origami edge-path is established by the bicorn path between the origami pair of curves. Additionally, we prove that the bicorn paths stay in a 14-neighborhood of any geodesic between an origami pair of curves. 

\begin{Def}
Let $\alpha, \beta \subset S_g$ be two simple closed curves that intersect minimally up to isotopy of, say $\alpha$.  A simple closed curve $\gamma$ is a \emph{bicorn curve}
between $\alpha$ and $\beta$ if either $\gamma=\alpha$, $\gamma=\beta$, or $\gamma$
is represented by the union of an arc $\alpha' \subset \alpha$ and an arc $\beta' \subset \beta$, which we call the $\alpha$-arc and the $\beta$-arc of $\gamma$, and $\alpha'$ only intersects $\beta'$ at the endpoints.
If $\gamma = \alpha$, then its $\alpha$-arc is $\alpha$ and its $\beta$-arc is empty, similarly if $\gamma = \beta$, then its $\beta$-arc is $\beta$ and its $\alpha$-arc is empty.
\end{Def}

\begin{figure}[ht]
\scalebox{.40}{\includegraphics[origin=c]{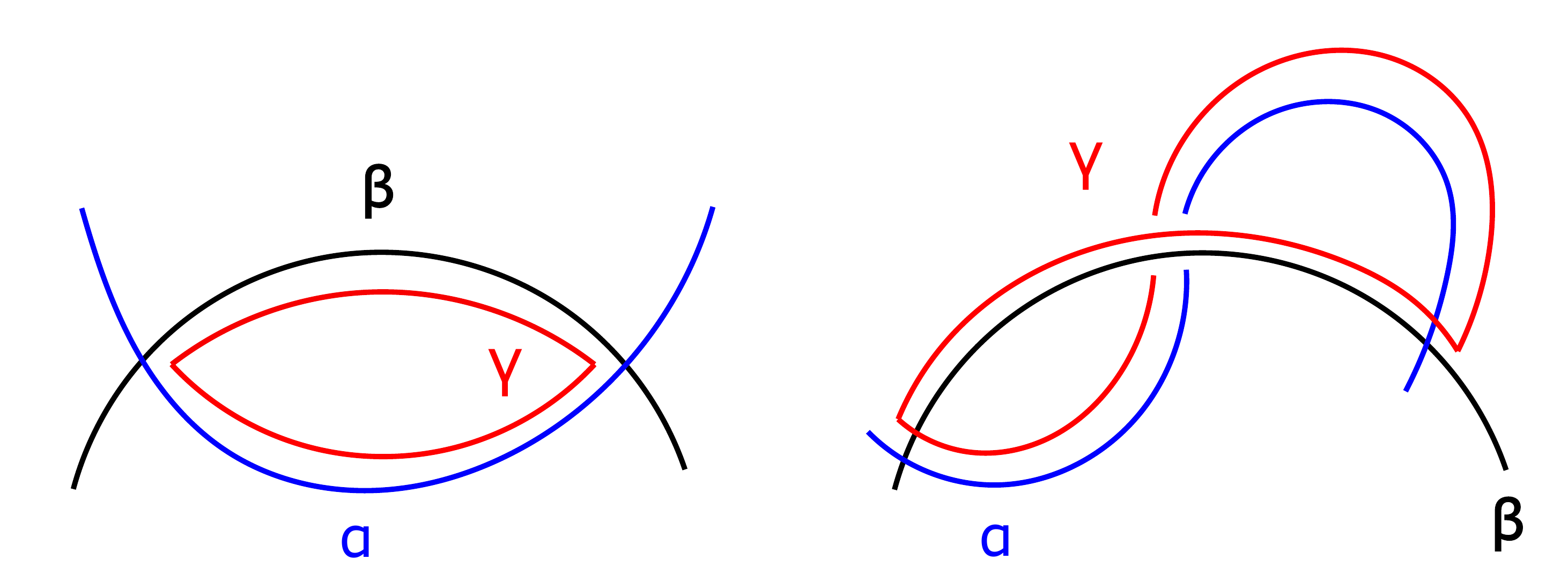}}
\caption{Two configurations of bicorn curves $\gamma$ between $\alpha$ and $\beta$. Only the right configuration occurs for coherent pairs. }
\label{bicorn}
\end{figure}

Based on the orientation of two intersection points, there are two configurations of the bicorn curves illustrated in Fig. \ref{bicorn}. Note that the bicorn curves are essential, as $\alpha$ and $\beta$ intersect minimally. Since the intersection number of $\alpha$ and $\beta$ is finite, then the number of bicorn curves is finite as well.

\begin{figure}[ht]
\scalebox{.50}{\includegraphics[origin=c, width=1.20 \textwidth,  height = 1.0 \textwidth]{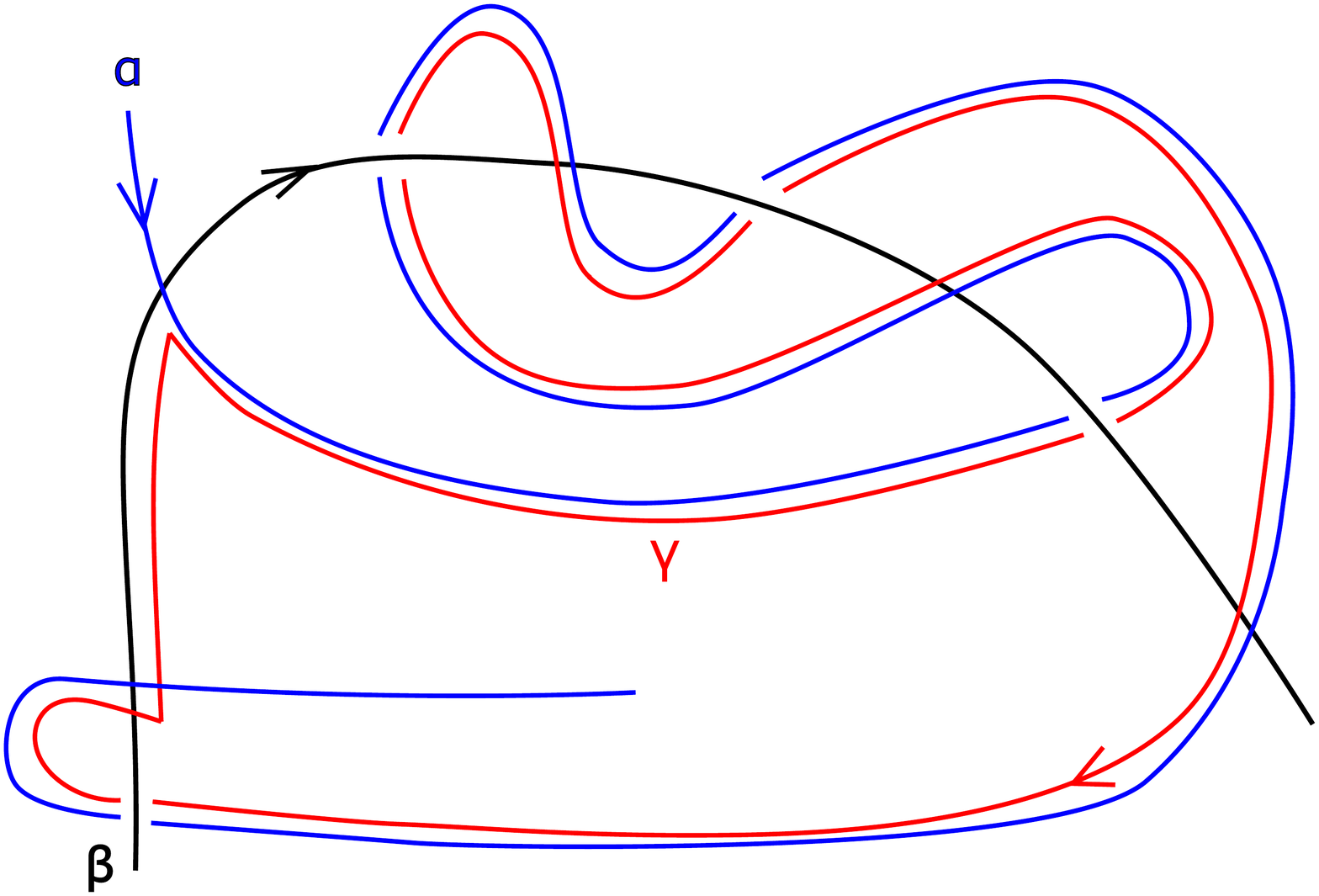}}
\vspace{0cm}
\caption{Bicorn curve $\gamma$ between $\alpha$ and $\beta$ with induced orientation from $\alpha$ and $\beta$.}
\label{bicorn_orientation}
\end{figure}

\begin{Prop}\label{intersection1}
For any two coherent non-separating curves $\alpha$, $\beta$ in the closed oriented surface $S_g$, there exists a sequence of non-separating curves $\alpha_0, \alpha_1, \alpha_2, \cdots, \alpha_n$ such that $|\alpha_i \cap \alpha_{i+1}| = 1$ and $\alpha=\alpha_0$, $\beta=\alpha_n$. Moreover, any two curves $\alpha_i$ and $\alpha_j$ are coherent.
\end{Prop}

\begin{proof}
The proof proceeds similarly as that in the Claim 3.4 in \cite{Ra-1} used to prove Theorem \ref{Ra}. Assume that two non-separating curves $\alpha$ and $\beta$ are coherent with specified orientations on them. Take a minimal subarc $b' \subset \beta$ that does not intersect $\alpha$ except for the endpoints, and denote the curve $\alpha_1=a' \cup b'$ as the union of the minimal subarc $b' \subset \beta$ and the subarc $a'$ of $\alpha$ determined by the endpoints. See the Fig. \ref{bicorn_orientation}. Since $b'$ is minimal, then $a'$ intersects $b'$ only at the endpoints, then $\alpha_1$ is a bicorn  curve. Note that only the right configuration of Fig. \ref{bicorn} can occur, so $|\alpha \cap \alpha_1| = 1$. The assumption that $\alpha$ is non-separating implies $\alpha_1$ is non-separating. 

Following the orientation of $\beta$, we extend the minimal subarc $b'$ to the next intersection point with $a'$. The extended subarc $b''$ of $\beta$ intersects $a'$ on the endpoint, the bicorn curve is denoted as $\alpha_2=a''\cup b''$. See the Fig. \ref{bicorn_extension}. As we can see, $\alpha_1$ intersects $\alpha_2$ exactly once, then $\alpha_2$ is also non-separating. Next, we extend the subarc $b''$ to the minimal subarc $b'''$ such that $b'''$ intersects $a''$ right on the endpoint, the subarc of $a''$ with bounded by the new intersection point is denoted by $a'''$. The bicorn curve $\alpha_3=a'''\cup b'''$ intersects $\alpha_2$ exactly once, but $\alpha_3$ can intersect $\alpha_1$ more than twice. 

Continuing in this way, we will be able to construct a sequence of bicorn curves $\alpha=\alpha_0, \alpha_1, \alpha_2, \cdots, \alpha_n$, where the adjacent curves $\alpha_i$, $\alpha_{i+1}$ intersect exactly once. Since the intersection number of $\alpha$ and $\beta$ is finite, the sequence must terminate at $\beta$, that is, $\alpha_n=\beta$. By induction, all the bicorn curves $\alpha_i$ are non-separating. The novelty of this sequence is that any $(\alpha_i, \alpha_j)$ pairing is coherent, because they are the union of the subarcs with induced orientation from $\alpha$ and $\beta$. This completes the proof. 
\end{proof}

\begin{figure}[ht]
\scalebox{.50}{\includegraphics[origin=c, width=1.20 \textwidth,  height = 1.0 \textwidth]{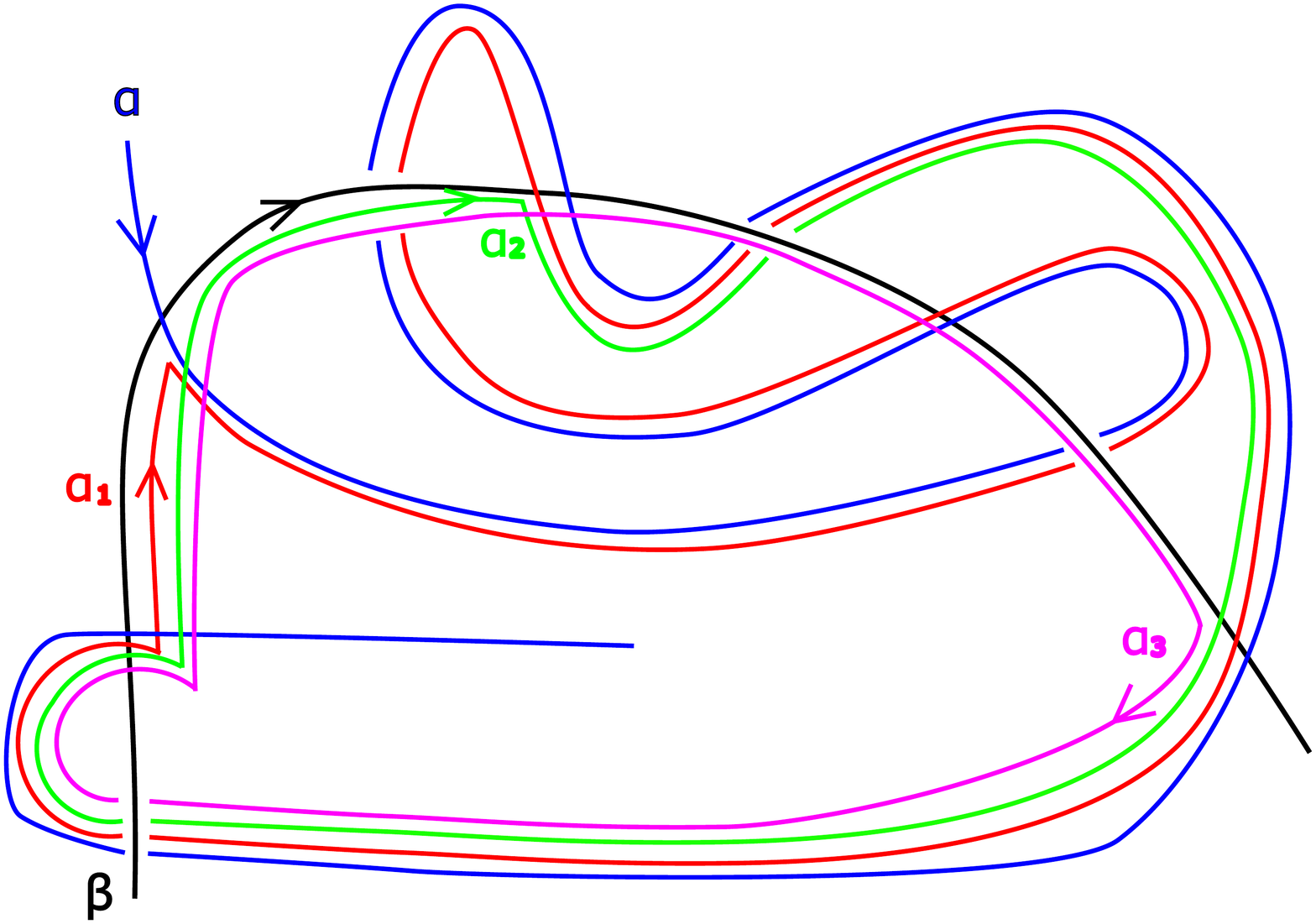}}
\vspace{0cm}
\caption{An example that extends bicorn curve $\alpha_1$ along $\beta$ to obtain bicorn curves $\alpha_2$ and $\alpha_3$ such that $|\alpha_2 \cap \alpha_3|=1$ and $|\alpha_1 \cap \alpha_3|=3$.}
\label{bicorn_extension}
\end{figure}

The sequence of bicorn curves constructed above is called a \emph{bicorn path}. It is not an actual path, as the adjacent curves are not necessarily disjoint. 

The Proposition \ref{intersection1} actually shows the following.

\begin{Thmn}[\ref{origami paths exist}]
For any origami pair $(u,v)$, there exists an origami edge-path $ \mathcal{E} = \{u = v_0 , \cdots , v_n = v \}$.
\end{Thmn}

In Lemma 2.4 of \cite{PS}, Przytycki and Sisto showed that the $2$-neighborhoods of bicorn paths between a pair of curves are connected. Bicorn paths are a generalization of the unicorn paths introduced in \cite{HPW} to prove the uniform hyperbolicity of arc graphs.  Proposition 4.2 in \cite{HPW} states that the unicorn paths stay in a $6$-neighborhood of a geodesic in the arc graph. We will show a similar result for bicorn paths in the curve graph. In the following lemma, we denote the set of all the bicorn curves between $\alpha$ and $\beta$ as $B(\alpha, \beta)$, and $\mathcal{B}\in B(\alpha,\beta)$ as a bicorn path. Note that bicorn curves are only representatives of curves, so they might not be distinct. Besides the endpoints, two bicorn paths can share vertices.  

\begin{Lem}\label{midpoint}
 Let $x_0,\cdots, x_m$ be a sequence of curves in $\mathcal{C}^1(S_g)$ with $2^{n-1}< m \leq 2^n$ for some positive integer $n$, then for any bicorn curve $c\in \mathcal{B} \in B(x_0, x_m)$, there exists a curve $c^*\in \mathcal{B^*}\in B(x_i, x_{i+1})$ such that $d(c, c^*) \leq 2n = 2\lceil log_2(m) \rceil$, where $\lceil * \rceil$ denotes the least integer that is larger than or equal to $*$. 
\end{Lem}

\begin{proof}
By Lemma 2.6 in \cite{PS}, each bicorn curve in $B(\alpha, \beta)$ is in the 2-neighborhood of $B(\alpha, \gamma)\cup B(\gamma, \beta)$ for any curve $\gamma$. In particular, we can take the $\gamma=x_{2^{n-1}}$. By induction, continue to take the midpoints $n-1$ times for the preceding pieces of segments, there must be a curve $c^*\in \mathcal{B^*}\in B(x_i,x_{i+1})$ for some $i$ such that $d(c,c^*)\leq 2n = 2\lceil log_2(m)\rceil $. 
\end{proof}

\begin{Prop}\label{14-bounded}
 Let $\Gamma$ be a geodesic connecting two curves $\alpha$ and $\beta$, then the bicorn paths between $\alpha$ and $\beta$ stay in a 14-neighborhood of $\Gamma$ in the curve graph $\mathcal{C} ^1(S_g)$. 
\end{Prop}

\begin{proof}
 The proof is almost identical to that of the Proposition 4.2 in \cite{HPW}. Suppose that the curve $c\in \mathcal{B}\in B(\alpha, \beta)$ is at maximal distance $k$ from the geodesic $\Gamma$. Choose such a bicorn path $\mathcal{B}$, and take the maximal bicorn subpath $\mathcal{B'}\subset \mathcal{B}$ containing $c$ such that the endpoints $a'$ and $b'$ of the subpath is at distance $\leq 2k$ from $c$. If $[c,a']$ or $[c,b']$ covers $[c,\alpha]$ or $[c,\beta]$, then let $a'=\alpha$ or $b'=\beta$. By the construction of bicorn curves, we know that $\mathcal{B'}\in \mathcal{B}(a',b')$. 
 
\begin{figure}[ht]
\vspace{-2cm}
\scalebox{.60}{\includegraphics[origin=c, width=1.20 \textwidth,  height = 1.0 \textwidth]{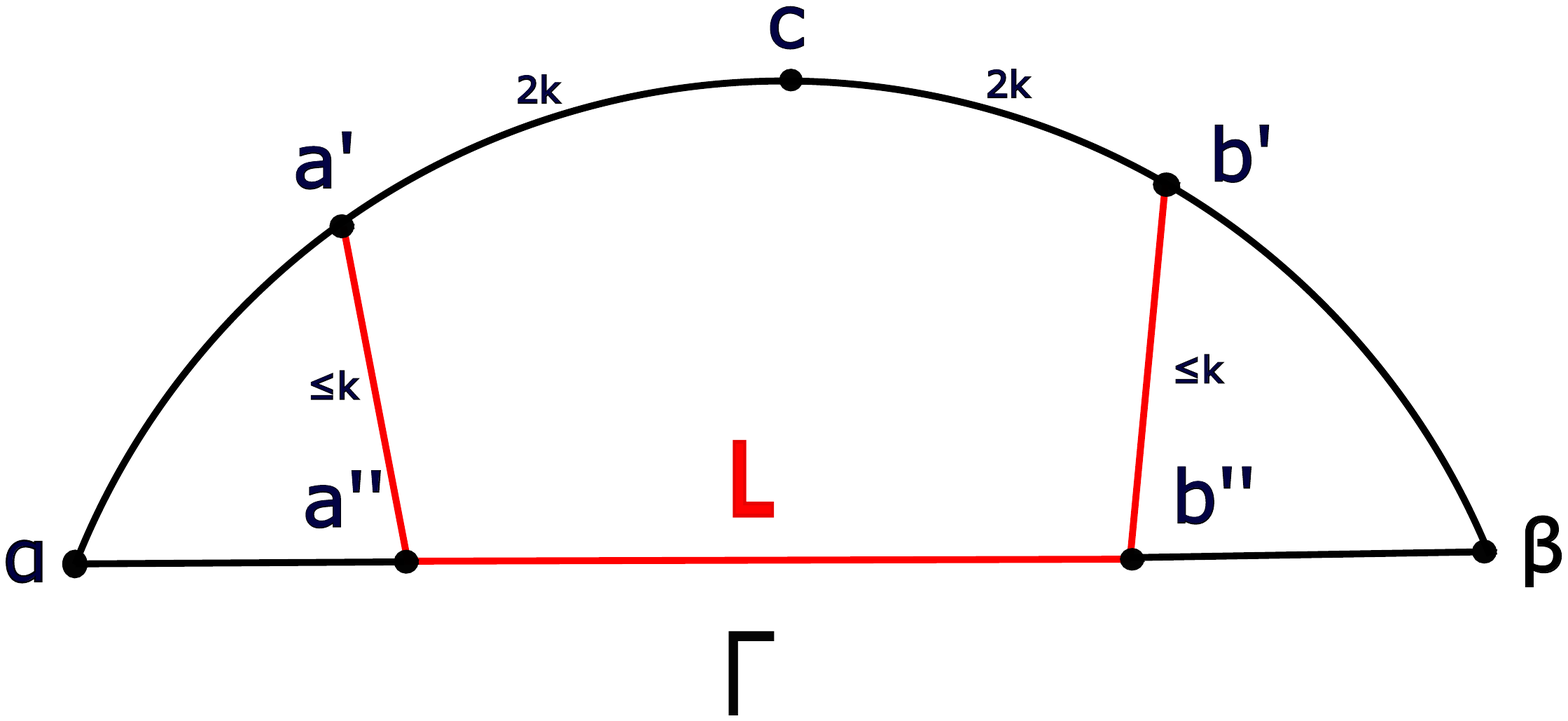}}
\vspace{-2cm}
\caption{A bicorn path $\mathcal{B}\in B(\alpha, \beta)$ is on the top and a geodesic $\Gamma=[\alpha, \beta]$ is at the bottom.}
\label{bounded}
\end{figure}
 
 Let $a''$ and $b''$ be the curves in the geodesic $\Gamma$ that are closest to the curves $a'$ and $b'$, then $d(a',a'') \leq k$ and $d(b',b'') \leq k$. If $a' = \alpha$ or $b' = \beta$, then $a'' = \alpha$ or $b'' = \beta$. See Fig. \ref{bounded}. By the triangle inequality, we have $d(a'',b'') \leq d(a'',a')+d(a',b')+d(b',b'') \leq k+4k+k=6k$. Let $L$ be the path constructed by concatenation of $[a'', b'']\subset \Gamma$ and some geodesics $[a',a'']$, $[b',b'']$. 
 
 Apply the previous Lemma \ref{midpoint} to the sequence of curves in the path $L$ and the curve $c\in \mathcal{B'}\subset \mathcal{B}(a',b')$. Note that the number of curves in $L$ is at most $k+6k+k=8k$. Therefore, the distance between curve $c$ and path $L$ is less than or equal to $2\lceil log_2(8k)\rceil$. Let $x_i$ be the curve of $L$ that obtains the distance, then either $x_i \in [a',a''],[b',b'']$ or $x_i\in [a'', b'']\subset \Gamma$. In the first case, $d(c,x_i) \geq d(c,a') - d(a',x_i) \geq k$ and $d(c,x_i)\geq d(c,b') - d(b',x_i)\geq k$. In the latter case, as $c$ is distance $k$ from $\Gamma$, then $k \leq 2 \lceil log_2(8k) \rceil$. It follows that $k \leq 2 \lceil log_2(8k) \rceil$ has integer solutions when $k\leq14$. 
 \end{proof}
 

\begin{figure}[ht]
\scalebox{.50}{\includegraphics[origin=c, width=1.50 \textwidth,  height = 1.0 \textwidth]{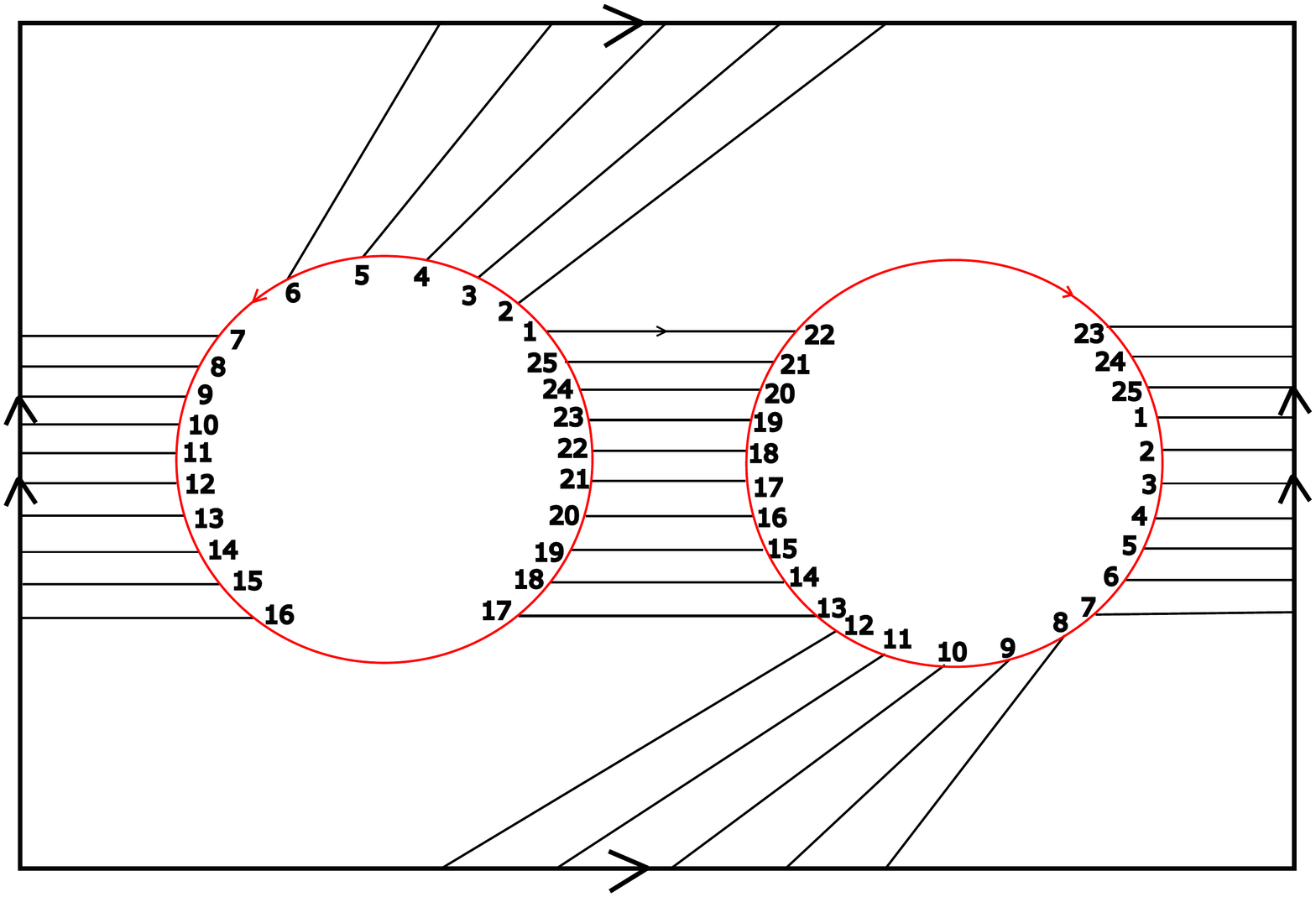}}
\vspace{-0cm}
\caption{Hempel's example is a coherent filling pair.}
\label{Hempel}
\end{figure}

\begin{figure}[hbt!]
\scalebox{.50}{\includegraphics[origin=c, width=1.50 \textwidth,  height = 1.0 \textwidth]{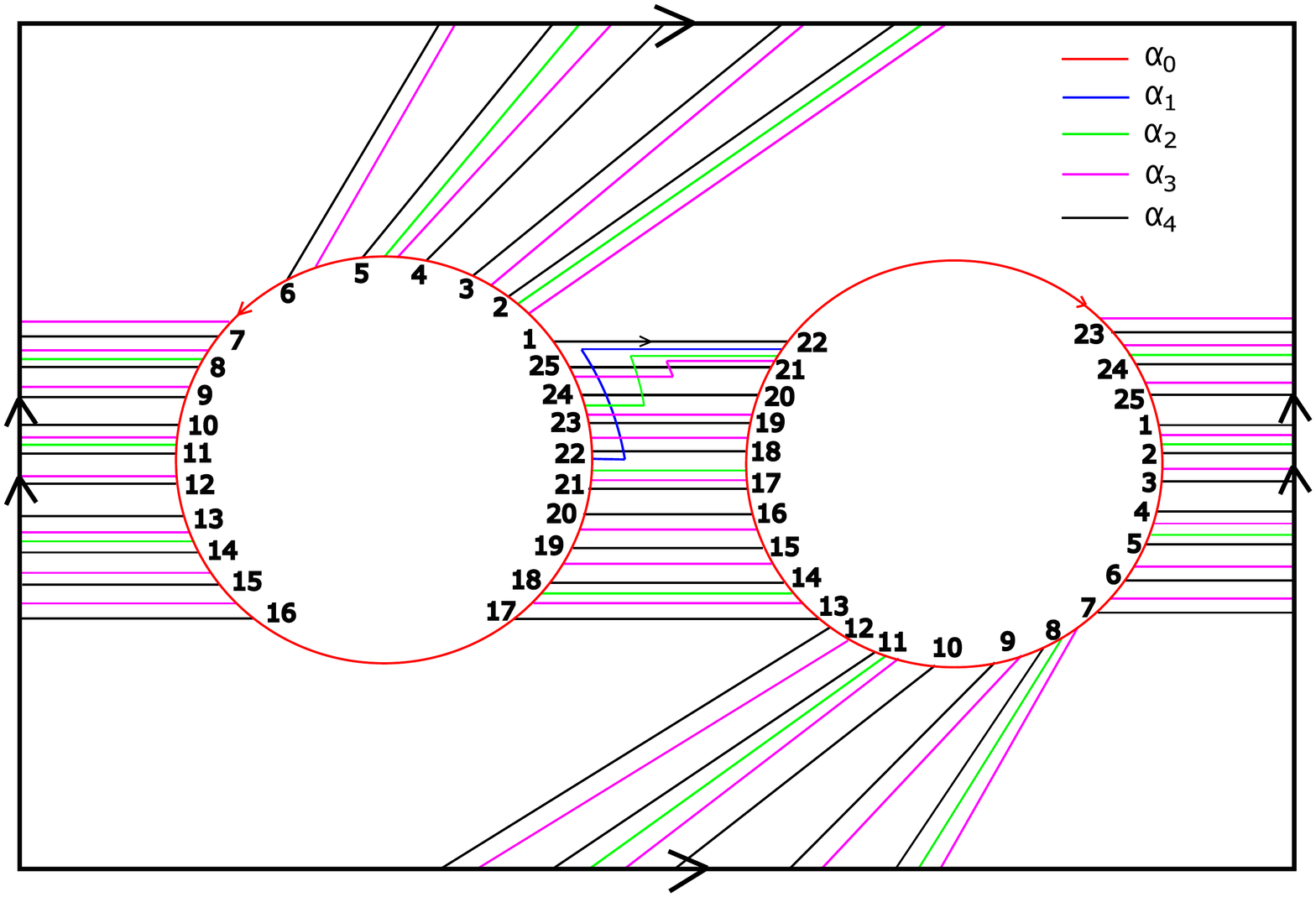}}
\vspace{-0cm}
\caption{An example of an origami edge-path in Hempel's example.}
\label{Hempel_path}
\end{figure}

\section{Examples}
\label{examples}

\subsection{Hempel's origami pair for genus $2$.}
\label{Hempel's example}

A well known example of a filling pair is Hempel's genus 2 distance 4 example \cite{Sc}.
The surface $S_2$ can be represented by a rectangle minus two discs as illustrated in Fig. \ref{Hempel}.  First, the boundaries of the two discs are glued with the indicated orientation and the indices on the boundaries are matched. For the rectangle, the left side is glued with the right side with the indicated orientation, and arcs intersecting with them are joined up in the same order. Similarly, the top and bottom sides are glued to match the arcs in the order as indicated. The result is a closed surface with genus 2 with two simple closed curves on it, one of which is red, and the other is black.  The resulting red curve as the boundaries of the discs glued together is denoted by $\alpha$. The black arcs are closed up as a simple closed curve, denoted by $\beta$.  We observe that this filling pair is coherent and, thus, we have an origami pair.  The distance between them in the curve graph is $4$ \cite{BMW, GMMM}.  In Fig. \ref{Hempel_path}, we find a path of bicorn curves in the Hempel's example by the bicorn curve construction, where $\alpha=\alpha_0$ and $\beta=\alpha_4$.

It is readily checked that the pairs $(\alpha_0, \alpha_3)$ and $(\alpha_1 , \alpha_4)$ are also origami pairs of curves.  Let $u_{\alpha}$ and $v_\beta$ be the associated vertices in $\ns$.
Then we have ${\bf E}(u_\alpha , v_\beta) \leq 4 = \dns(u_\alpha , v_\beta)$.  Additionally, subjecting this bicorn path to the intersection quotient progression of Conjecture \ref{Conj2}
we see the following:
$$\frac{| \alpha_4 \cap \alpha_1 |}{| \alpha_0 \cap \alpha_1 |} \left(= \frac{4}{1} \right) >  \frac{| \alpha_4 \cap \alpha_2 |}{| \alpha_0 \cap \alpha_2 |} \left(= \frac{2}{8} \right) > \frac{| \alpha_4 \cap \alpha_3 |}{| \alpha_0 \cap \alpha_3 |} \left(= \frac{1}{19} \right).$$
In fact, by construction all bicorn paths will satisfy such a decreasing progression.  Our Conjecture \ref{Conj2} would imply that all bicorn paths are quasi-geodesics.

\subsection{Origami pairs of curves for $S_{g>2}$.}
For Euler characteristic reasons the theoretical minimal number of intersections required for a pair of curves in $S_{g \geq 2}$ to be a filling pair is $2g-1$.
For $g>2$, this minimal intersection is geometrically realizable \cite{AH}, and such a pair of curves is called a \emph{minimally intersecting filling pair}.  When restricting to coherent filling pairs their existence is established by the following result, as well as the work of Jeffreys \cite{Jef} on [1,1]-origamis. 

\begin{Th}\emph{(Aougab-Menasco-Nieland \cite{AMN})}
For any closed oriented surface $S_g$ with $g > 2$, there exists a minimally intersecting filling pair that intersects coherently. 
\end{Th}

The complement of a minimally intersecting filling pair, $(\alpha , \beta )$, of $S_{g>2}$ has the property that $S \setminus (\alpha \cup \beta)$ is a single
$4(2g-1)$-gon disc.  As such, finding an $\alpha_1 \subset S_g$ having $|\alpha \cap \alpha_1 | = 1 = |\beta \cap \alpha_1 |$ is straightforward. In the single $4(2g-1)$-gon, choose two $\alpha$-subarcs and two $\beta$-subarcs that share the same edge, and connect them by two arcs. These two arcs are either disjoint or intersecting once, as illustrated in Fig. \ref{curves_origami}. In the first case, a band operation creates the $\alpha_1$. Otherwise,  take $\alpha_1$ to be the curve by reducing the intersection point. 

\begin{figure}[ht]
\vspace{-1cm}
\scalebox{.50}{\includegraphics[origin=c, width=1.50 \textwidth,  height = 1.0 \textwidth]{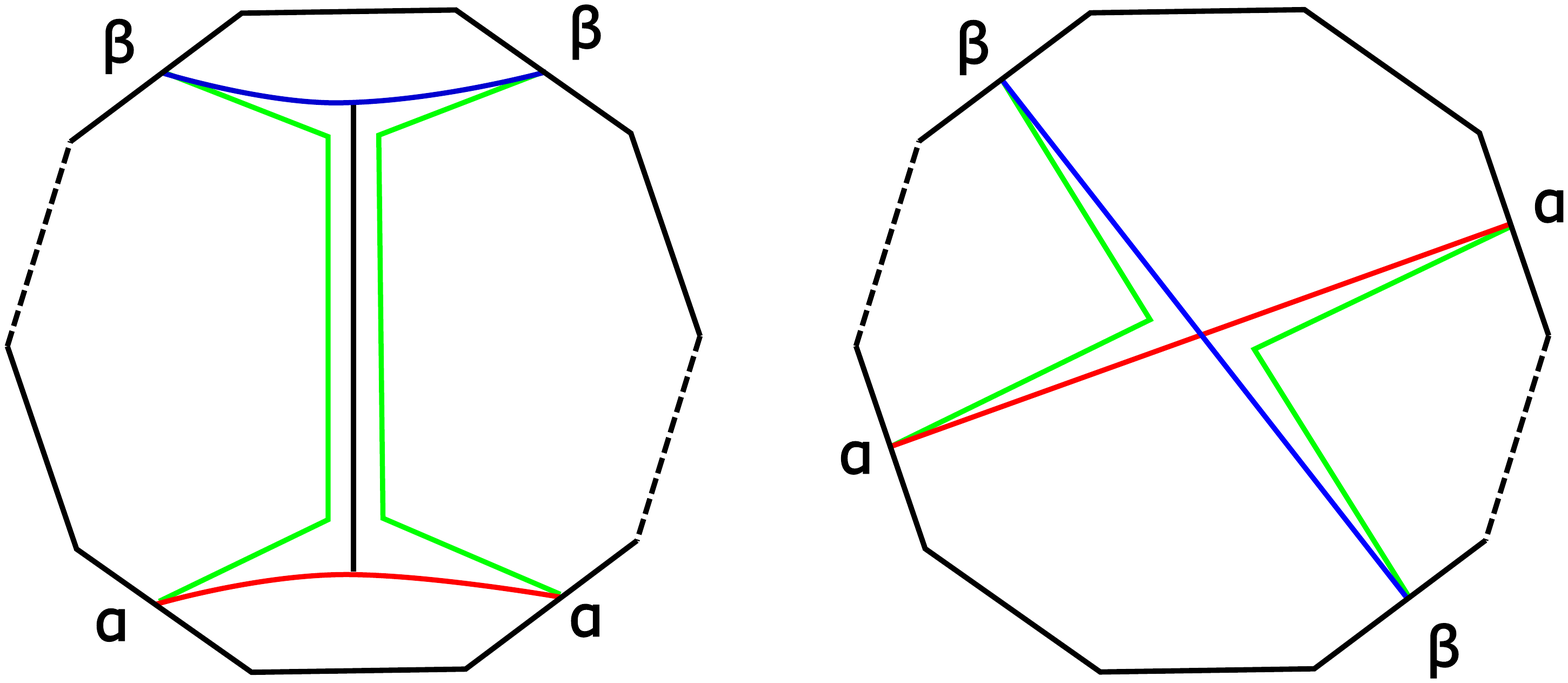}}
\vspace{-1cm}
\caption{ Two arcs in the left figure are disjoint; Two arcs in the right figure intersect exactly once. The arc connecting two $\alpha$-subarcs is in red; The arc connecting two $\beta$-subarcs is in blue; $\alpha_1$ is in green. }
\label{curves_origami}
\end{figure}


\section{Arbitrarily large distance of coherent filling pair}
\label{arbitrary distance}

In this section we establish Theorem \ref{infinite diameter}, that an origami pair of curves can be arbitrarily distant in the non-separating curve graph $\ns$. The basic strategy is to use a particular pseudo-Anosov to map a curve far away from its original position. Meanwhile, the resulting curve remains in coherent intersection. To achieve that, we need the following property of pseudo-Anosov maps. 

\begin{Th}\emph{(Masur-Minsky \cite{MM1} Proposition 4.6)}
\label{MM_pS}
For a surface $S_{g,b}$ with the complexity $\xi(S_{g,b})>1$, there exists $c>0$ 
such that, for any pseudo-Anosov $h\in Mod(S_{g,b})$, any $\gamma \in 
\mathcal{C}(S_{g,b})$ and any $n\in \mathbb{Z}$,
$$d_{\mathcal{C}(S_{g,b})}(h^n(\gamma),\gamma)\geq c|n|.$$
\end{Th}

For a closed oriented surface $S_g$ with $g\geq 2$, the complexity $\xi(S_g)>3$, then the above theorem can be applied to this case. Another ingredient is Thurston's construction of pseudo-Anosov homeomorphisms used to preserve coherent intersection. 
\begin{Th}\emph{(Thurston \cite{Thurston} Theorem 7)}
\label{Thurston}
Suppose that two curves $\alpha$ and $\beta$ is a filling pair for $S_g$, then any product of positive powers of the Dehn twists $T_{\alpha}$ and negative powers of Dehn twists $T_{\beta}$ is a pseudo-Anosov homeomorphism, provided that each $T_{\alpha}$ and $T_{\beta}^{-1}$ occurs at least once. 
\end{Th}

Before we start the proof, we need to show the existence of an origami pair of curves on a closed oriented surface $S_g$.

\begin{Lem}
 Given a closed oriented surface $S_g$ of genus $g\geq 2$, there exists a coherent filling pair on it. 
\end{Lem}
\begin{proof}
 We can use Hempel's example for the genus 2 case. 
 When $g>2$, Aougab-Menasco-Nieland's construction \cite{AMN} of minimally intersecting filling pair gives a coherent filling pair. 
\end{proof}

\begin{Prop}\label{large}
For any non-separating curve $\alpha$, there is another non-separating curve $\beta$ intersecting coherently with $\alpha$, and the distance between them can be taken to be arbitrarily large in the curve graph $\mathcal{C}^1(S_g)$.
\end{Prop}

\begin{proof}

\begin{figure}[ht]
\vspace{0cm}
\scalebox{.60}{\includegraphics[origin=c, width=1.20 \textwidth,  height = 1.0 \textwidth]{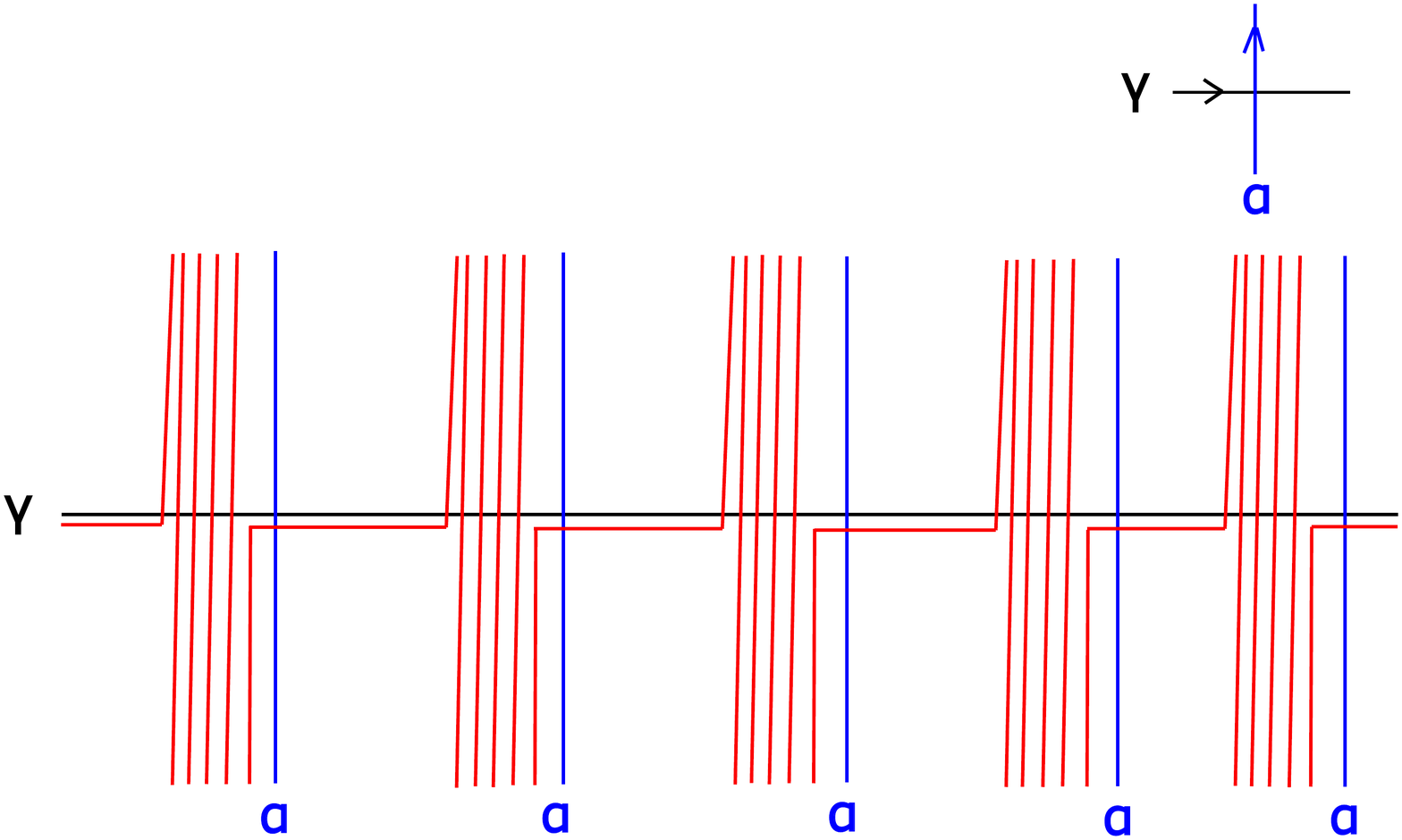}}
\vspace{-1cm}
\caption{Dehn twist $T_{\alpha}(\gamma)$ is represented by curve surgery, where $\alpha$ and $\gamma$ is a coherent pair and $|\alpha \cap \gamma|=5$.}
\label{curve_surgery}
\end{figure}

Take $\alpha$ and $\gamma$ to be a coherent filling pair from the previous lemma. Let $h=T_{\alpha}\circ T_{\gamma}^{-1}$, then $h(\gamma)=T_{\alpha}\circ T_{\gamma}^{-1}(\gamma)=T_{\alpha}(\gamma)$. It can be observed that $T_{\alpha}(\gamma)$ is in coherent intersection with $\alpha$ and $\gamma$ by using the curve surgery to picture the Dehn twist. In Fig. \ref{curve_surgery}, the intersection number $|\alpha \cap \gamma|=5$. The horizontal black line is the curve $\gamma$ and the vertical blue lines are the arcs in the curve $\alpha$. 
 
To start off, we take a horizontal parallel copy of the curve $\gamma$, and $|\alpha \cap \gamma|$ parallel copies of $\alpha$ with same orientation. At each intersection point, if we follow the arc of $\gamma$ towards the intersection with orientation indicated in the Fig. \ref{curve_surgery}, it goes to the left at the intersection then follow along the arc of $\alpha$ toward the next intersection, we go to the right when we hit a copy of $\gamma$. 
It is straightforward to see that the resulting curve $T_{\alpha}(\gamma)$ is in coherent intersection with $\gamma$, as well as $\alpha$. 

By the same curve surgery, $T_{\gamma}^{-1}(T_{\alpha}(\gamma))$ coherently intersects with $\alpha$ and $\gamma$, so does $T_{\alpha}\circ T_{\gamma}^{-1}(T_{\alpha}(\gamma))$. It follows that $h^2(\gamma)$ is in coherent intersection with $\alpha$ and $\gamma$. By induction, $h^n(\gamma)$ coherently intersects $\alpha$ and $\gamma$. Let $h^n(\gamma)$ be $\beta$, then 
$d_{\mathcal{C}(S_{g})}(\alpha, \beta)=d_{\mathcal{C}(S_{g})}(\alpha, h^n(\gamma))\geq  d_{\mathcal{C}(S_{g})}(\gamma,h^n(\gamma)) - d_{\mathcal{C}(S_{g})}(\alpha, \gamma)$. By Thurston's construction in Theorem \ref{Thurston}, $h$ is a pseudo-Anosov homeomorphism. Using Theorem \ref{MM_pS}, we show that 
$d_{\mathcal{C}(S_{g})}(\gamma,h^n(\gamma))\rightarrow \infty$, as $n\rightarrow \infty$. That implies $d_{\mathcal{C}(S_{g})}(\alpha, \beta)$ is arbitrarily large.
\end{proof}

Since the non-separating curve graph $\ns$ is quasi-isometric to the curve graph $\mathcal{C}^1(S_g)$, Theorem \ref{infinite diameter} follows immediately from Proposition \ref{large}.

\end{document}